\documentclass[11pt]{amsart}
\usepackage{amsmath, amssymb, latexsym}
\usepackage{mathrsfs}
\usepackage{enumerate}
\usepackage[all, knot]{xy}
\usepackage{color}
\usepackage{bm}
\usepackage{graphicx}

\newtheorem{thm}{Theorem}[section]
\newtheorem{prop}[thm]{Proposition}
\newtheorem{cor}[thm]{Corollary}
\newtheorem{lem}[thm]{Lemma}

\theoremstyle{definition}

\theoremstyle{remark}
\newtheorem{Rmk}[thm]{Remark}

\newenvironment{red}
{\relax\color{red}}
{\hspace*{.5ex}\relax}

\newcommand{\ber}{\begin{red}}
\newcommand{\er}{\end{red}}

\newenvironment{verd}
{\relax\color{magenta}}
{\hspace*{.5ex}\relax}

\newcommand{\bg}{\begin{verd}}
\newcommand{\eg}{\end{verd}}

\numberwithin{equation}{subsection}

\newcommand{\Z}{\mathbb{Z}}
\newcommand{\Q}{\mathbb{Q}}

\newcommand{\g}{\mathfrak{g}}

\newcommand{\Hom}{\mathrm{Hom}}

\newcommand{\residue}{{\rm res}}

\newcommand{\Ind}{{\rm Ind}}
\newcommand{\Res}{{\rm Res}}
\newcommand{\Top}{{\rm Top}}

\newcommand{\Rad}{{\rm Rad}}

\newcommand{\rlQ}{\mathsf{Q}}   
\newcommand{\wlP}{\mathsf{P}}   
\newcommand{\weyl}{\mathsf{W}}  
\newcommand{\cmA}{\mathsf{A}}  

\newcommand{\ST}{\mathsf{ST}}   

\newcommand{\bR}{\mathbf{k}}   
\newcommand{\fqH}{R^{\Lambda}}   
\newcommand{\fqHH}{R^{2\Lambda_0}}

\newcommand{\sBox}[1]
{
\xy
(-2,-2)*{};(2,-2)*{} **\dir{-};
(-2,2)*{};(2,2)*{} **\dir{-};
(-2,-2)*{};(-2,2)*{} **\dir{-};
(2,-2)*{};(2,2)*{} **\dir{-};
(0,0)*{#1};
\endxy
}

\begin{document}

\title[Representation type for blocks of Hecke algebras of classical type]
{Representation type for block algebras of Hecke algebras of classical type}

\author[Susumu Ariki]{Susumu Ariki}
\thanks{S.A. is supported in part by JSPS, Grant-in-Aid for Scientific Research (C) 15K04782.}
\address{Department of Pure and Applied Mathematics, Graduate School of Information
Science and Technology, Osaka University, 1-5 Yamadaoka, Suita, Osaka 565-0871, Japan}
\email{ariki@ist.osaka-u.ac.jp}


\begin{abstract}
We find representation type of the cyclotomic quiver Hecke algebras $\fqH(\beta)$ of level two in affine type $A$. 
In particular, we have determined representation type for all the block algebras of Hecke algebras of classical type 
(except for characteristic two in type $D$), which has not been known for a long time. 
As an application of this result, we prove that block algebras of finite representation type 
are Brauer tree algebras whose Brauer trees are straight lines without exceptional vertex if the Hecke algebras 
are of classical type. We conjecture that this statement should hold for Hecke algebras of exceptional type.
\end{abstract}

\maketitle


\vskip 2em

\section*{Introduction}

Hecke algebras of classical type are important in modular representation theory of finite groups of Lie type in non-defining characteristic. For example, 
those of type $A$ and $B$ appear as endomorphism algebras in Harish-Chandra series. Because of the importance, they 
have received detailed studies since 1980's when Dipper and James initiated the field.

A long standing open problem in modular representation theory of Hecke algebras was to determine representation type of Hecke algebras of type $B$ with equal or 
unequal parameters blockwise. For Hecke algebras of type $A$, representation type of their block algebras was determined by 
Erdmann and Nakano \cite{EN02}.\footnote{The proof in \cite{EN02} uses a result which judged representation type by complexity, 
but this result is known to be false. However, the authors had a different proof and they were able to avoid its use in the second proof. 
Thus, their final result in \cite{EN02} holds without any change. See explanation in the corrigendum to \cite{Ar05}.} 
Another proof is given by the author, Iijima and Park \cite{AIP13}. In that paper, we have determined representation type for a one parameter family of 
finite quiver Hecke algebras in affine type $A$, and the family includes block algebras of Hecke algebras in type $A$ 
at a special parameter value.\footnote{Warning for those who are not familiar with Hecke algebras. The parameter $q$ which appears in the 
definition of the Hecke algebra of type $A$ is fixed and it determines the affine Lie type $A^{(1)}_\ell$. The one parameter we mention 
here is another parameter which comes from polynomials $Q_{ij}(u,v)$ in the definition of the finite quiver Hecke algebra in affine type $A$.}
On the other hand, we may embed Hecke algebras of type $D$ to Hecke algebras of type $B$. Assuming that the base field has an odd characteristic, 
we may apply Clifford theory to determine representation type of block algebras of Hecke algebras in type $D$. See Appendix 2.
Therefore, the problem to solve is to determine representation type of block algebras of Hecke algebras in type $B$. 

Let $\bR$ be an algebraically closed field, $q\ne1, Q\in \bR^\times$. The Hecke algebra of type $B$ is the $\bR$-algebra $\mathcal{H}_n(q,Q)$ defined by generators 
$T_0,T_1,\dots,T_{n-1}$ and relations
\begin{gather*}(T_0-Q)(T_0+1)=0,\quad (T_i-q)(T_i+1)=0\;\;(1\le i<n)\\
T_0T_1T_0T_1=T_1T_0T_1T_0, \quad T_iT_{i+1}T_i=T_{i+1}T_iT_{i+1}\;\;(1\le i<n-1)\\
T_iT_j=T_jT_i\;\;(j\ne i\pm1).
\end{gather*}
We have two cases to consider. When $-Q\not\in q^{\Z}$, then Morita equivalence theorem by Dipper and James \cite{DJ92} implies that we determine representation type of an (outer)  
tensor product of two block algebras of Hecke algebras of type $A$. Then, it suffices to consider tensor product of block algebras of finite or tame representation type, as 
representation types for other tensor products are clear. Block algebras of finite representation type are Morita equivalent to the principal block of rank $\ell+1$ \cite{U92}. 
Further, block algebras of tame representation type appear only when $\ell=1$ \cite{EN02}, and they are Morita equivalent to either 
the principal block of rank $4$ or $5$ by Scopes' Morita equivalence. Therefore, we can determine the representation type blockwise when $-Q\not\in q^{\Z}$ by studying 
those rather simple tensor product algebras. See Appendix 1. 
Hence we focus on the case $-Q=q^s$, for some $s\in\Z$, and the Hecke algebra of type $B$ in this case is isomorphic to the cyclotomic Hecke algebra associated with level two dominant weight $\Lambda=\Lambda_0+\Lambda_s$. 

In my work in 1990's, I developed Fock space theory for cyclotomic Hecke algebras, and sophisticated machinery was adopted to calculate decomposition numbers 
in characteristic zero \cite{Ar01}. 
With some trick, it also allowed us to compute some crucial decomposition numbers in positive characteristic for certain block algebras \cite{AM04}. Using this computation, 
I generalized and settled Uno's conjecture which predicted the representation type of the whole algebra in terms of the Poincar\'e polynomial of the Weyl group \cite{Ar05}.  
However, this strategy does not work for determining representation type blockwise. After some years I had a new idea when cyclotomic quiver Hecke algebras were introduced and 
Brundan and Kleshchev proved that cyclotomic Hecke algebras are isomorphic to cyclotomic quiver Hecke algebras in affine type $A$ \cite{BK09(a)} and gave 
graded version of the author's Fock space theory \cite{BK09(b)}. 
The advantage of cyclotomic quiver Hecke algebras is that they have many easily visible idempotents, 
so that we may control the structure of various spherical subalgebras and show that 
they have wild representation type. Then, we focus on small number of cases remained, and it is reasonable to expect that we may determine representation type for these cases as well. 

In a series of papers with Park \cite{AP12}, \cite{AP13}, \cite{AP14} and with Iijima and Park \cite{AIP13},  we applied the above strategy to finite quiver Hecke algebras in type $A^{(2)}_{2\ell}, D^{(2)}_{\ell+1}, C^{(1)}_\ell$ and $A^{(1)}_\ell$ with parameters $\lambda$. 
Using the techniques we developed through the experience, it is now not difficult to determine representation type of Hecke algebras of type $B$ blockwise: we consider a slightly larger class of algebras, namely cyclotomic quiver Hecke algebras $\fqH(\beta)$ of level two in affine type $A$ because they give block algebras of the Hecke algebras of type $B$ 
as a special case \cite{LM07}. 
Then, by algebra automorphisms induced by Dynkin automorphisms and Chuang and Rouquier's derived equivalences among Weyl group orbits of weights \cite{CR08}, it suffices to 
consider 
$$
\Lambda=\Lambda_0+\Lambda_s, \;\;\text{for}\;\; 0\le s\le\ell, \;\;\text{and} \;\; \beta=\lambda^s_i+k\delta, \;\;\text{for}\;\; 0\le i\le \frac{\ell-s+1}{2}, 
$$
where $\lambda^s_i=\sum_{k=0}^s i\alpha_k + \sum_{k=1}^{i-1} (i-k)\alpha_{s+k} + \sum_{k=1}^{i-1} k\alpha_{\ell-i+1+k}$. Here, we understand that
$\lambda^s_0=0$, $\lambda^s_1=\alpha_0+\cdots+\alpha_s$. 

\medskip
In this paper we consider the case when $s\ne0$ and prove the following theorem, which is Theorem \ref{main thm} in the body of the paper.

\medskip
\noindent
\textbf{Theorem A}\;\;
Let $\Lambda=\Lambda_0+\Lambda_s$, for $1\le s\le\ell$, and $\beta=\lambda^s_i+k\delta$, for $0\le i\le \frac{\ell-s+1}{2}$ and $k\in\Z_{\ge0}$. Then 
$\fqH(\beta)$ is
\begin{itemize}
\item[(i)]
simple if $i=0$ and $k=0$.
\item[(ii)]
of finite representation type if $i=1$ and $k=0$. 
\item[(iii)]
of tame representation type if $i=0$, $k=1$ and $\ell=1$.
\item[(iv)]
of wild representation type otherwise.
\end{itemize}

For $s=0$, my student Kakei has computed the representation type \cite{Ka15} and the result is as follows. Recall that $\fqH(\beta)$ in affine type $A$ has a parameter $\lambda$.
\footnote{As he left mathematics without writing a paper to publish, I got his permission to add an appendix to this paper for making the computation 
available to anyone who wants to check the result.  See Appendix 3.}

\medskip
\noindent
\textbf{Theorem B}\;\;
Let $\Lambda=2\Lambda_0$ and $\beta=\lambda^0_i+k\delta$, for $0\le i\le \frac{\ell+1}{2}$ and $k\in\Z_{\ge0}$. Then 
$\fqH(\beta)$ is
\begin{itemize}
\item[(i)]
simple if $i=0$ and $k=0$.
\item[(ii)]
of finite representation type if $i=1$ and $k=0$. 
\item[(iii)]
of tame representation type if $\ell=1, i=0, k=1$ or $\ell\ge2, i=0, k=1, \lambda\ne(-1)^{\ell+1}$, or $\ell\ge3, i=2, k=0, {\rm char}\bR\ne2$. 
\item[(iv)]
of wild representation type otherwise.
\end{itemize}

He uses arguments from the aforementioned papers by the author and Park. In this paper, we utilize grading of the algebra more systematically for proving Theorem A. 

Further, if we combine Theorem A and Theorem B with Ohmatsu's result, whose proof 
is given in \cite{AKMW}, and with a result by Geck \cite{G}, we can prove the following.

\medskip
\noindent
\textbf{Theorem C}\;\;
Let $\bR$ be an algebraically closed field of odd characteristic, $q\ne 1, Q\in\bR^\times$ as before, and let
$\mathcal{H}$ be a Hecke algebra of classical type, $B$ a block algebra of $\mathcal{H}$. 
If $B$ is of finite representation type, then $B$ is a Brauer tree algebra whose Brauer tree 
is a straight line without exceptional vertex.

\bigskip
We conjecture that Theorem C should hold if $\mathcal{H}$ is a Hecke algebra of exceptional type and 
bad primes are invertible in $\bR$. 
One supporting evidence in type $E$ may be found in \cite{G2}.

\section{Preliminaries}

In this section, we briefly recall necessary materials. See \cite{HK02} and \cite{AIP13} for more details.

\subsection{Cartan datum} \label{Sec: Cartan datum}

Let $I = \{0,1, \ldots, \ell \}$ be an index set, and let $\cmA=(a_{ij})_{i,j,\in I}$ be the Cartan matrix of type $A_{\ell}^{(1)}$. 
We fix an affine Cartan datum $(\cmA, \wlP, \Pi, \Pi^{\vee})$, where
\begin{itemize}
\item[(1)] $\wlP$ is the weight lattice, a free abelian group of rank $\ell+1$.
\item[(2)] $\Pi = \{ \alpha_i \mid i\in I \} \subset \wlP$, the set of simple roots.
\item[(3)] $\Pi^{\vee} = \{ h_i \mid i\in I\} \subset \wlP^{\vee} := \Hom( \wlP, \Z )$, the set of simple coroots.
\end{itemize}
By definition, the simple roots and the simple coroots satisfy the following properties:
\begin{itemize}
\item[(a)] $\langle h_i, \alpha_j \rangle  = a_{ij}$ for all $i,j\in I$,
\item[(b)] $\Pi$ and $\Pi^{\vee}$ are linearly independent sets.
\end{itemize}
Then, the root lattice and its positve cone are defined by
$$
\rlQ = \sum_{i \in I} \Z \alpha_i, \quad \rlQ^+ = \sum_{i\in I} \Z_{\ge 0} \alpha_i,
$$
and the Weyl group, which we denote by $\weyl$, is the affine symmetric group. It is generated by $\{r_i\}_{i\in I}$, where
$r_i\Lambda=\Lambda-\langle h_i, \Lambda\rangle\alpha_i$, for $\Lambda\in \wlP$. The null root is
$$ \delta = \alpha_0 + \alpha_1 + \cdots + \alpha_\ell. $$

\subsection{Quantum group}

We fix a scaling element $d$ which obeys the condition $\langle d, \alpha_i\rangle=\delta_{i0}$,
and assume that $\Pi^{\vee}$ and $d$ form a $\Z$-basis of $\wlP^{\vee}$. Then, we define $U_{\bf q}(\g)$, the quantum group associated with the affine Cartan datum, as usual: it is
generated by ${\bf q}^h$, for $h\in \wlP^\vee$, and the Chevalley generators $e_i, f_i$, for $i\in I$, which obey the quantum Serre relations etc.

\subsection{Integrable highest weight modules}

We define fundamental weights $\Lambda_j$ $(j\in I)$ by
$$\langle h_i, \Lambda_j\rangle=\delta_{ij},\quad \langle d,\Lambda_j\rangle=0.$$
Then $\{\delta, \Lambda_0,\dots, \Lambda_\ell\}$ is a $\Z$-basis of $\wlP$. Moreover, we have a symmetric bilinear form $(\hphantom{-}| \hphantom{-})$ on $\wlP$ 
given by the Cartan matrix $\cmA$. 
We only use the integrable highest weight $U_{\bf q}(\g)$-modules 
$V_{\bf q}(\Lambda_0+\Lambda_s)$, for $0\le s\le\ell$, in our paper. 

\subsection{Cyclotomic quiver Hecke algebras}

Let $\bR$ be an algebraically closed field. 
We take polynomials $Q_{ij}(u,v)\in\bR[u,v]$ in the standard form, for $i,j\in I$, as follows. 
$$
Q_{01}(u,v)=u^2+\lambda uv+v^2,
$$
for the affine type $A^{(1)}_1$, and
\begin{align*}
Q_{ij}(u,v)&=u+v\;\;(j=i+1, 0\le i\le \ell-1),\\
Q_{\ell0}(u,v)&=u+\lambda v,\\
Q_{ij}(u,v)&=1\;\;(j\not\equiv i\pm1\mod(\ell+1)),
\end{align*}
for the affine type $A^{(1)}_\ell$ with $\ell\ge2$. Here $\lambda\in\bR$ is a parameter such that $\lambda\ne0$ unless $\ell=1$.
Then, we define the {\it cyclotomic quiver Hecke algebra} $\fqH(n)$ associated with
$(Q_{ij}(u,v))_{i,j\in I}$. It is a $\Z$-graded  $\bR$-algebra defined by generators
$e(\nu)$, for $\nu = (\nu_1,\ldots, \nu_n) \in I^n$, $x_1,\dots,x_n$, $\psi_1,\dots, \psi_{n-1}$, and their relations, and the 
$\Z$-grading is given by
\begin{align*}
\deg(e(\nu))=0, \quad \deg(x_k e(\nu))= ( \alpha_{\nu_k} |\alpha_{\nu_k}), \quad  \deg(\psi_l e(\nu))= -(\alpha_{\nu_{l}} | \alpha_{\nu_{l+1}}).
\end{align*}

For $\beta\in \rlQ^+$ which is a sum of $n$ simple roots, we define a central idempotent $e(\beta)$ by
$$
e(\beta) = \sum_{\nu \in I^\beta} e(\nu),\quad \text{where}\;\;
I^\beta = \left\{ \nu=(\nu_1, \ldots, \nu_n ) \in I^n \mid \sum_{k=1}^n\alpha_{\nu_k} = \beta \right\}.
$$
Then we denote $R^\Lambda(\beta) = R^\Lambda(n) e(\beta)$.

\begin{Rmk}
If $Q_{ij}(u,v)$, for $i\ne j$, are nonzero constant multiples of $(u-v)^{-a_{ij}}$ as in \cite{BK09(a)}, 
then $\lambda=(-1)^{\ell+1}$ if $\ell\ge2$, and $\lambda=-2$ if $\ell=1$, in the standard form.
Thus, $\fqH(\beta)$ for this choice of the parameter $\lambda$ and $\Lambda=\Lambda_0+\Lambda_s$ are the block algebras of $\mathcal{H}(q,-q^s)$.
\end{Rmk}

\subsection{The level two Fock space} \label{Sec: Fock space}

Let $\mathcal{F}_{\bf q}$ be the vector space over $\Q({\bf q})$ whose basis is given by
$\{ |\lambda \rangle \mid \lambda=(\lambda^{(1)}, \lambda^{(2)}) : \text{bipartition} \}$. We make bipartitions $I$-colored by declaring a residue pattern on nodes of 
bipartitions as follows. 
\begin{itemize}
\item If a node $x\in\lambda^{(1)}$ sits on the $r$-th row and the $c$-th column of $\lambda^{(1)}$, the residue of $x$ is $i\in I$ such that $i+r-c$ is divisible by $\ell+1$.
\item If a node $x\in\lambda^{(2)}$ sits on the $r$-th row and the $c$-th column of $\lambda^{(2)}$, the residue of $x$ is $i\in I$ such that $i+r-c-s$ is divisible by $\ell+1$.
\end{itemize}

Then $\mathcal{F}_{\bf q}$ is a $U_{\bf q}(\g)$-module and the empty bipartition generates a submodule which is isomorphic to $V_{\bf q}(\Lambda_0+\Lambda_s)$. 
To describe the action on $\mathcal{F}_{\bf q}$, we use the following notation:
\begin{itemize}
\item If we may remove a box of residue $i$ from $\lambda$ and obtain a new bipartition, then we write $\lambda \nearrow \sBox{i}$ for the resulting bipartition.
\item If we may add a box of residue $i$ to $\lambda$ and obtain a new bipartition, then we write $\lambda \swarrow \sBox{i}$ for the resulting bipartition.
\end{itemize}
We declare that nodes of $\lambda^{(1)}$ are above those of $\lambda^{(2)}$ in $\lambda=(\lambda^{(1)},\lambda^{(2)})$. 
Then, we define, for $\mu=\lambda \nearrow \sBox{i}$, $d_b(\lambda/\mu)$ to be the number of addable $i$-nodes below $\lambda/\mu$ minus the number of 
removable $i$-nodes below $\lambda/\mu$. Similarly, we define, for  $\mu=\lambda \swarrow \sBox{i}$, $d^a(\mu/\lambda)$ to be the number of addable $i$-nodes above $\mu/\lambda$ minus the number of removable $i$-nodes above $\mu/\lambda$. Then, the Chevalley generators act on  $\mathcal{F}_{\bf q}$ as follows. 
$$
e_i  |\lambda\rangle= \sum_{\mu=\lambda \nearrow \sBox{i}} {\bf q}^{d_b(\lambda/\mu)} |\mu\rangle, \quad
f_i  |\lambda\rangle= \sum_{\mu=\lambda \swarrow \sBox{i}} {\bf q}^{-d^a(\lambda/\mu)} |\mu\rangle.
$$

\subsection{Graded dimension formula for $\fqH(\beta)$}

Let $\ST(\lambda)$ be the set of standard bitableaux of shape $\lambda$. If $\lambda=(\lambda^{(1)},\lambda^{(2)})$ is a bipartition with $n$ nodes, 
a standard bitableau of shape $\lambda$ is a filling of nodes of $\lambda$ with $1,2,\dots,n$ such that 
\begin{itemize}
\item[(i)]
Each of $1,\dots,n$ is used exactly once.
\item[(ii)]
The numbers increase left to right in each row of $\lambda^{(1)}$ and $\lambda^{(2)}$.
\item[(iii)]
The numbers increase top to bottom in each column of $\lambda^{(1)}$ and $\lambda^{(2)}$.
\end{itemize}
For $T\in \ST(\lambda)$, we view it as an increasing sequence of bipartitions, and suppose that nodes $x_1,\dots, x_n$ are added in this order in the process of increasing 
the bipartitions. Then, we define the degree and the residue sequence of $T$ by
$$ \deg(T)=\sum_{i=1}^n d_b(x_i), \quad \residue(T)=(\residue(x_1),\dots,\residue(x_n)). $$
The next formula is proved in the entirely similar manner as \cite[Thm.4.20]{BK09(b)}.

\begin{thm} \label{Thm: dimension formula}
For $\nu, \nu' \in I^n$, we have
$$
\dim_{\bf q} e(\nu')\fqH(n)e(\nu) = \sum_{\lambda \vdash n} K_{\bf q}(\lambda, \nu') K_{\bf q}(\lambda, \nu),
$$
where $K_{\bf q}(\lambda, \nu)$ is the sum of $ {\bf q}^{\deg(T)}$ over $T\in \ST(\lambda)$ with $\residue(T)=\nu$.
\end{thm}

\begin{lem}
Suppose that $e=e_1+e_2$ with $e_1^2=e_1\ne 0$, $e_2^2=e_2\ne 0$, $e_1e_2=e_2e_1=0$ and
$$
\dim_{\bf q} e_i\fqH(n)e_j-\delta_{ij}-c_{ij}{\bf q}^2\in {\bf q}^3\Z_{\ge0}[{\bf q}], 
$$
for $i,j=1,2$. Then the quiver of $e\fqH(n)e$ has two vertices $1$ and $2$, $c_{ii}$ loops on the vertex $i$, for $i=1,2$, and there are at least $c_{12}$ arrows and 
$c_{21}$ reverse arrows between $1$ and $2$.
\end{lem}
\begin{proof}
Let $A=e\fqH(n)e$. Our assumption implies that $A$ is positively graded, so that strictly positive elements form $\Rad A$. Further, $e_iAe_i$, for $i=1,2$, are local algebras. Thus, 
$c_{ij}$ degree two elements give linearly independent elements of $e_i(\Rad A/\Rad^2 A)e_j$.
\end{proof}

\section{Reduction to certain maximal weights}

Recall that we use cyclotomic quiver Hecke algebras of level two instead of Hecke algebras of type $B$, because of the Brundan-Kleshchev 
isomorphism theorem \cite{BK09(a)}. 
As is mentioned in the introduction, the strategy to determine representation type of cyclotomic quiver Hecke algebras was tested in our series of papers 
\cite{AP12}, \cite{AP13}, \cite{AP14} and \cite{AIP13}. In the next two paragraphs, we briefly recall the strategy following \cite{AP12}. 

Let $\Lambda=\Lambda_0+\Lambda_s$ be a level two dominant integral weight for affine Lie type $A^{(1)}_\ell$. Then our 
categorification theorem together with a result by Lyle and Mathas \cite{LM07} asserts that 
weight spaces of the integrable highest weight module $V(\Lambda)$ over the affine Lie algebra 
$\mathfrak{g}(A^{(1)}_\ell)$, where $\ell+1$ is the quantum characteristic of $q$, are categorified by the block algebras of $\mathcal{H}(q,-q^s)$, or more generally, 
by cyclotomic quiver Hecke algebras $\fqH(\beta)$, for $\beta\in\rlQ^+$. 
Recall that these algebras are self-injective algebras, so that if two of them are derived equivalent then they are stably equivalent. 
Thus, they have the common representation type by Krause's theorem. 

Using the derived equivalence by Chuang and Rouquier, which lifts the Weyl group action, the above consideration implies that 
it suffices to consider representatives of
$\weyl$-orbits in the set of weights $\{ \mu \in \wlP \mid V(\Lambda)_{\mu}\neq 0\}$, 
in order to determine the representation type of $\fqH(\beta)$, for all $\beta=\Lambda-\mu\in\rlQ^+$. Define
\begin{align*}
\lambda^s_i&=\sum_{k=0}^s i\alpha_k + \sum_{k=1}^{i-1} (i-k)\alpha_{s+k} + \sum_{k=1}^{i-1} k\alpha_{\ell-i+1+k}, \\
\mu^s_i&=\sum_{k=0}^{i-1} (i-k)\alpha_k + \sum_{k=1}^{i-1} k\alpha_{s-i+k} + \sum_{k=1}^{\ell-s+1} i\alpha_{s-1+k}.
\end{align*}

We set $\lambda^s_0=0$. The following result may be found in \cite[Thm.1.3]{TW15}.

\begin{lem}\label{orbit representative}
Let $\weyl$ be the affine symmetric group. Then, representatives of $\weyl$-orbits in $\{ \mu\in \wlP \mid V(\Lambda_0+\Lambda_s)_\mu\neq 0\}$ are given by
$$
\Lambda_0+\Lambda_s-\lambda^s_i-k\delta \;\;(0\leq i\leq\frac{\ell+1-s}{2}),\;\;\; 
\Lambda_0+\Lambda_s-\mu^s_i-k\delta \;\;(1\leq i\leq\frac{s}{2}), \;\;\;
$$
where $k\in\Z_{\ge0}$.
\end{lem}

\section{Isomorphisms induced by Dynkin automorphisms}

\begin{lem}\label{isom lemma}
Let $\sigma: I\simeq I$ be a Dynkin automorphism, namely a bijective map that satisfies $a_{\sigma(i)\sigma(j)}=a_{ij}$, for $i,j\in I$. For 
$\beta=\sum_{i\in I} b_i\alpha_i\in \rlQ^+$ and $\Lambda=\sum_{i\in I} c_i\Lambda_i\in \wlP^+$, we define 
$$
\sigma\beta=\sum_{i\in I} b_i\alpha_{\sigma(i)}, \quad \sigma\Lambda=\sum_{i\in I} c_i\Lambda_{\sigma(i)}.
$$
Then, $\fqH(\beta)$ defined with $Q_{ij}(u,v)$ is isomorphic to $R^{\sigma\Lambda}(\sigma\beta)$ defined with $Q'_{ij}(u,v)$ such that
$$
Q'_{\sigma(i)\sigma(j)}(u,v)=Q_{ij}(u,v).
$$
\end{lem}
\begin{proof}
We check that the assignment $e(\nu_1\cdots\nu_n)\mapsto e(\sigma(\nu_1)\cdots\sigma(\nu_n))$, $\psi_k\mapsto \psi_k$, $x_k\mapsto x_k$ defines an isomorphism of algebras.
\end{proof}

\begin{prop}\label{Dynkin autom}
Suppose that we are in the affine type $A^{(1)}_\ell$ and consider
$$
\sigma:\Z/(\ell+1)\Z\simeq \Z/(\ell+1)\Z
$$
given by $\sigma(i)=i+1$. Then, using the same $Q_{ij}(u,v)$ in the 
standard form, we have an isomorphism of algebras $R^\Lambda(\beta)\simeq R^{\sigma\Lambda}(\sigma\beta)$.
\end{prop}
\begin{proof}
First we suppose that $\ell\ge2$ and let $Q'_{ij}(u,v)$, for $i\ne j$, be such that $Q'_{01}(u,v)=u+\lambda v$, $Q'_{ij}(u,v)=u+v$ if $j=i+1, 1\le i\le \ell$ and  $Q'_{ij}(u,v)=1$ otherwise, Then, 
Lemma \ref{isom lemma} implies that it suffices to show that the standard form of $Q'_{ij}(u,v)$ is $Q_{ij}(u,v)$, that is, we show the existence of a symmetric 
matrix $(c_{ij})_{i,j\in I}$ which satisfies
$$
Q'_{ij}(u,v)=c_{ij}^2Q_{ij}(c_{ii}u,c_{jj}v), \;\; c_{ij}\ne0, \;\text{for all $i,j\in I$}.
$$
It is satisfied by choosing $c_{00}=1$, $c_{ii}=\lambda$, for $1\le i\le \ell$, and 
$$
c_{ij}=\begin{cases}
1 \quad &(i=0, j=1) \\
1/\sqrt{\lambda} \quad &(j=i+1, 1\le i\le \ell)\\
1 \quad &(\text{otherwise}).\end{cases}
$$
The proof for the case $\ell=1$ follows from the similar argument.
\end{proof}

By Proposition \ref{Dynkin autom}, it suffices to consider $\Lambda=\Lambda_0+\Lambda_s$, for $0\le s\le\ell$, 
in order to study representation type of $\fqH(\beta)$ for any level two $\Lambda$. Further, we have the following corollary.

\begin{cor}\label{mu and lambda}
Suppose that we are in the affine type $A^{(1)}_\ell$ and $\Lambda=\Lambda_0+\Lambda_s$, for $1\le s\le \ell$. Then we have the following isomorphism of algebras
$$
\fqH(\mu^s_i+k\delta)\simeq \fqH(\lambda^{\ell-s+1}_i+k\delta), \;\;\text{for $1\le i\le \frac{s}{2}$ and $k\in \Z_{\ge0}$.}
$$
\end{cor}

By Corollary \ref{mu and lambda}, it suffices to consider $\fqH(\lambda^s_i+k\delta)$, for $1\le i\le \frac{\ell-s+1}{2}$, to determine the representation type of $\fqH(\beta)$.

\section{Representation type for $i=0$}

\begin{prop}\label{i=0}
Suppose that $\Lambda=\Lambda_0+\Lambda_s$ with $1\le s\le \ell$. Then $\fqH(\delta)$ is of wild representation type unless $\ell=1$. 
If $\ell=1$ then $\fqH(\delta)$ is of tame representation type. 
\end{prop}
\begin{proof}
We consider $A=e\fqH(\delta)e$, for $e=e_1+e_2$, where $e_1$ and $e_2$ are given by 
$$
e_1=e(0,1,\dots,\ell), \;\; e_2=e(s,s+1,\dots,\ell,0,1,\dots,s-1).
$$
Then, the graded dimension formula shows 
\begin{align*}
\dim_{\bf q} e_i\fqH(\delta)e_i&=1+c_i{\bf q}^2+{\bf q}^4, \;\;\text{for $i=1, 2$,}\\
\dim_{\bf q} e_1\fqH(\delta)e_2&=\dim_{\bf q} e_2\fqH(\delta)e_1={\bf q}^2,
\end{align*}
such that $c_1=2$ or $c_2=2$ holds if and only if $\ell\ge2$. 
Therefore, $e\fqH(\delta)e_1$ and  $e\fqH(\delta)e_2$ are pairwise non-isomorphic indecomposable projective $A$-modules, and 
if $\ell\ge2$ then the quiver of $A$ has two loops on one of the two vertices $1$ and $2$ and there are arrows $1\rightarrow 2$ and 
$1\leftarrow 2$. Thus, $\fqH(\delta)$ is wild if $\ell\ge2$. If $\ell=1$ then $A=\fqH(\delta)$ and $c_1=c_2=1$. We may determine the 
radical series of $Ae_i$, for $i=1,2$, from graded dimensions $\dim_{\bf q} e_iAe_j$ and it follows that 
$A$ is the Brauer graph algebra for the graph $1-2-3$ with two exceptional vertices $1$ and $3$ with multiplicity $2$. Hence, $\fqH(\delta)$ is tame if $\ell=1$.
\end{proof}

\begin{prop}\label{ell=1}
If $\ell=1$ and $\Lambda=\Lambda_0+\Lambda_1$, then $\fqH(2\delta)$ is of wild representation type.
\end{prop}
\begin{proof}
We consider $A=e\fqH(2\delta)e$, for $e=e_1+e_2$, where $e_1$ and $e_2$ are given by 
$$
e_1=e(0101), \;\; e_2=e(1010).
$$
Then, the graded dimension formula shows 
\begin{align*}
\dim_{\bf q} e_1\fqH(2\delta)e_1&=\dim_{\bf q} e_2\fqH(2\delta)e_2=1+2{\bf q}^2+2{\bf q}^4+2{\bf q}^6+{\bf q}^8,\\
\dim_{\bf q} e_1\fqH(2\delta)e_2&=\dim_{\bf q} e_2\fqH(2\delta)e_1={\bf q}^2+2{\bf q}^4+{\bf q}^6.
\end{align*}
Thus, the quiver of $A$ has two loops on each vertex $1$ and $2$ and there are arrows $1\rightarrow 2$ and $1\leftarrow 2$. Thus, $A$ is wild and so is 
$\fqH(2\delta)$.
\end{proof}

\section{Representation type for $i=1$}

\begin{prop}\label{i=1}
Suppose that $\Lambda=\Lambda_0+\Lambda_s$ with $1\le s\le \ell$. 
Then $\fqH(\lambda^s_1)$, for $\ell\geq s+1$, is of finite representation type.
\end{prop}
\begin{proof}
As $\lambda^s_1=\alpha_0+\cdots+\alpha_s$, we may enumerate nonzero $e(\nu)$'s and they are 
$$
e_i=e(0,1,\dots,i-1,s,s-1,\dots,i+1,i), \;\;\text{for $0\le i\le s+1$}. 
$$
Then, graded dimensions are given as follows. 
$$
\dim_{\bf q} e_i\fqH(\lambda^s_1)e_j=\begin{cases} 1+{\bf q}^2 \quad &\text{(if $j=i$)}, \\
                                                             {\bf q} \quad &\text{(if $j=i\pm 1$)}, \\
                                                             0 \quad &\text{(otherwise)}. \end{cases}
$$
Hence, $\Rad \fqH(\lambda^s_1)$ is spanned by homogeneous elements of positive degree, $P_i=\fqH(\lambda^s_1)e_i$ are indecomposable and pairwise non-isomorphic. 
Further, as cyclotomic quiver Hecke algebras are symmetric algebras \cite[Appendix A]{SVV14}, we may conclude that the radical series of $P_i$ are
$$
P_0=\begin{array}{c} S_0 \\ S_1 \\ S_0 \end{array}, \quad
P_i=\begin{array}{c} S_i \\ S_{i-1}\oplus S_{i+1} \\ S_i \end{array}\;\text{($1\le i\le s$)}, \quad
P_{s+1}=\begin{array}{c} S_{s+1} \\ S_s \\ S_{s+1} \end{array}
$$
where $S_i=\Top(P_i)$, and $\fqH(\lambda^s_1)$ is a Brauer tree algebra. 
\end{proof}

\section{Representation type for $i=2$}

\begin{prop}\label{i=2}
Suppose that $\Lambda=\Lambda_0+\Lambda_s$ with $1\le s\le \ell$. 
Then $\fqH(\lambda^s_2)$, for $\ell\geq s+3$, is of wild representation type.
\end{prop}
\begin{proof}
We consider $A=e\fqH(\delta)e$, for $e=e_1+e_2$, where $e_1$ and $e_2$ are given by 
\begin{align*}
e_1&=e(0,1,\dots,s-2,s-1,s, s+1,\ell,0,\dots,s-2,s-1,s), \\
e_2&=e(0,1,\dots,s-2,s,s-1,s+1,\ell,0,\dots,s-2,s,s-1).
\end{align*}
Then, the graded dimension formula shows 
\begin{align*}
\dim_{\bf q} e_1\fqH(\delta)e_1&=\dim_{\bf q} e_2\fqH(\delta)e_2=1+2{\bf q}^2+{\bf q}^4, \\
\dim_{\bf q} e_1\fqH(\delta)e_2&=\dim_{\bf q} e_2\fqH(\delta)e_1={\bf q}^2.
\end{align*}
Now, the proof goes in the same manner as that for Proposition \ref{i=0}. 
\end{proof}

\section{Proof of the main theorem}

The following Lemma \ref{wildness propagation 1} and Lemma \ref{wildness propagation 2} are proved by the argument we repeatedly used in our series of papers with Euiyong Park. 
This argument is another crucial ingredient of our strategy and it is based on two results. For the reader's convenience, we recall the results. The first one is a general result.

\begin{prop}
\label{reduction to critical rank}
Let $A$ and $B$ be finite dimensional $\bR$-algebras and suppose that
there exist a constant $C>0$ and functors
$F:A\text{\rm -mod}\rightarrow B\text{\rm -mod}$, $E:B\text{\rm -mod} \rightarrow A\text{\rm -mod}$ 
such that, for any $A$-module $M$,
\begin{itemize}
\item[(i)]
$M$ is a direct summand of $EF(M)$ as an $A$-module,
\item[(ii)]
$\dim F(M)\le C\dim M$.
\end{itemize}
Then, the wildness of $A$ implies the wildness of $B$. 
\end{prop}

The second one states that the 
induction functors $F_i$ and the restriction functors $E_i$ between module categories of cyclotomic quiver Hecke algebras 
may serve as the functors $F$ and $E$ in Proposition \ref{reduction to critical rank}. We state their result \cite{KK11} and \cite{Kash11} 
in ungraded form. Note that if $l_i>0$ then the conditions (i) and (ii) are satisfied.

\begin{thm}
Set $l_i = \langle h_i, \Lambda - \beta  \rangle$, for $i\in I$. Then one of the following isomorphisms of endofuctors on
the category of finitely generated $R^{\Lambda}(\beta)$-modules holds.
\begin{enumerate}
\item If $l_i \ge 0$, then $F_i E_i \oplus |l_i|\mathrm{id} \buildrel \sim \over \longrightarrow E_iF_i$.
\item If $l_i \le 0$, then $F_i E_i  \buildrel \sim \over \longrightarrow  E_iF_i \oplus |l_i|\mathrm{id}$.
\end{enumerate}
\end{thm}

Now we are ready to prove Lemma \ref{wildness propagation 1} and Lemma \ref{wildness propagation 2}.

\begin{lem}\label{wildness propagation 1}
Suppose $1\le i\le \frac{\ell-s+1}{2}$ and $k\in\Z_{\ge0}$. If $\fqH(\lambda^s_i+k\delta)$ is of wild representation type, so is 
$\fqH(\lambda^s_{i-1}+(k+1)\delta)$.
\end{lem}
\begin{proof}
We add simple roots $\alpha_{s+i},\dots, \alpha_{\ell-i+1}$ to $\lambda^s_i$ in this order. We may check
$$
\langle h_{s+i+k}, \Lambda_0+\Lambda_s-\lambda^s_1-\alpha_{s+i}-\cdots-\alpha_{s+i+k-1} \rangle\ge 1, 
$$
for $0\le k\le \ell-s-2i+1$. 
To carry out this computation, it is helpful to arrange the simple roots in $\lambda^s_i$ on the horizontally long rectangle diagram ($i$ rows and $s+i$ columns) as follows.
\[
\begin{array}{ccccccc}
\alpha_0 & \alpha_1 & \cdots & \alpha_s & \cdots & \cdots & \alpha_{s+i-1} \\
\alpha_\ell & \alpha_0 & \alpha_1 & \ddots & \ddots & \ddots & \vdots \\
\vdots & \ddots & \ddots & \ddots & \ddots & \ddots & \vdots \\
\alpha_{\ell-i+2} & \cdots & \alpha_\ell & \alpha_0 & \alpha_1 & \cdots & \alpha_s \\
\end{array}
\]
Then we add the simple roots to the first row of the diagram in the order $\alpha_{s+i},\dots, \alpha_{\ell-i+1}$.
Once this inequality is verified, we may apply Proposition \ref{reduction to critical rank} in each step of adding simple roots, and the result follows.
\end{proof}

\begin{lem}\label{wildness propagation 2}
Suppose $0\le i\le \frac{\ell-s-1}{2}$ and $k\in\Z_{\ge0}$. If $\fqH(\lambda^s_i+k\delta)$ is of wild representation type, so is 
$\fqH(\lambda^s_{i+1}+k\delta)$.
\end{lem}
\begin{proof}
We add simple roots
$$
\alpha_{s+i},\alpha_{s+i-1},\dots,\alpha_{s+1};\alpha_{\ell-i+1},\alpha_{\ell-i+2},\dots,\alpha_\ell,\alpha_0, \alpha_1,\dots,\alpha_{s-1}; \alpha_s
$$
to $\lambda^s_i$ in this order. Then, the rest of the proof goes in the similar manner as Lemma \ref{wildness propagation 1}.
\end{proof}

\begin{thm}\label{main thm}
Let $\Lambda=\Lambda_0+\Lambda_s$, for $1\le s\le\ell$, and $\beta=\lambda^s_i+k\delta$, for $0\le i\le \frac{\ell-s+1}{2}$ and $k\in\Z_{\ge0}$. Then 
$\fqH(\beta)$ is
\begin{itemize}
\item[(i)]
simple if $i=0$ and $k=0$.
\item[(ii)]
of finite representation type if $i=1$ and $k=0$. 
\item[(iii)]
of tame representation type if $i=0$, $k=1$ and $\ell=1$.
\item[(iv)]
of wild representation type otherwise.
\end{itemize}
\end{thm}
\begin{proof}
(i) is clear since $\fqH(0)=\bR$. (ii) is Propsosition \ref{i=1}. (iii) is proved in Proposition \ref{i=0}. Using Lemma \ref{wildness propagation 1} and Lemma \ref{wildness propagation 2}, 
(iv) follows from Proposition \ref{i=0}, Proposition \ref{ell=1} and Proposition \ref{i=2}.
\end{proof}

\noindent
By Lemma \ref{orbit representative} and Corollary \ref{mu and lambda}, Theorem \ref{main thm} and the similar result for $s=0$ by Kakei tell representation type of $R^{\Lambda_0+\Lambda_s}(\beta)$, 
for all $\beta\in\rlQ^+$ and for all $0\le s\le\ell$. Moreover, we have obtained explicit representatives for the derived equivalence classes 
of finite and tame $\fqH(\beta)$'s in the course of the proof.

\section{Proof of Theorem C}

As an application of the main result, we can prove Theorem C. 
Let ${\mathcal H}$ be a Hecke algebra of classical type over an algebraically closed field $\bR$ of odd characteristic 
and $B$ a block algebra of ${\mathcal H}$. We suppose that 
$B$ is of finite representation type. It is well-known that Hecke algebras of type $A$ and type $B$ 
are cellular algebras. We also know that Hecke algebras of type $D$ are cellular by \cite[Theorem 1.1]{G}. Hence 
$B$ is a symmetric cellular algebra of finite type. Let $e$ be the quantum characteristic of $q$. As $q\ne 1$, $e$ is the multiplicative order of $q$.

\subsection{}
If ${\mathcal H}$ is of type $A$, then 
$B$ is a Brauer tree algebra whose Brauer tree is a straight line without exceptional vertex. 
This follows from the radical series of indecomposable $B$-modules given in Appendix 1.

\subsection{}
Let $T_0,T_1,\dots,T_{n-1}$ be the standard generators of ${\mathcal H}$. The quadratic relations are 
$$
(T_0-Q)(T_0+1)=0, \quad (T_i-q)(T_i+1)=0\;\; (1\le i< n).
$$
If $-Q\not\in q^{\Z}$, then $B$ is the tensor product of block algebras $B_1$ and $B_2$ of type $A$, such that 
one is of finite type and the other is simple. Hence $B$ is Morita equivalent to $B_1$ or $B_2$, so that $B$ is 
a Brauer tree algebra whose Brauer tree is a straight line without exceptional vertex. 
Suppose that $-Q=q^s$. 
If $Q\ne -1$, we apply Theorem A. If $e$ is even and $Q=-1=q^{e/2}$,  
we apply Theorem B. 
For the algebras of finite representation type that appear in Theorem A, see the proof of Propsosition \ref{i=1}. 
The algebra of finite representation type that appears in Theorem B is $\bR[x]/(x^2)$.  
In the both cases, $B$ is derived equivalent to a Brauer tree algebra whose Brauer tree is a straight line without exceptional vertex. 
Therefore, applying Rickard's star theorem \cite[Theorem 4.2]{R89},  $B$ is derived equivalent to 
the star-shape Brauer tree algebra whose central vertex is not exceptional. (By the statement of the star theorem, 
the only vertex which could be exceptional is the central vertex.)

On the other hand, Ohmatsu's theorem, which is proved in \cite[Proposition 3.14]{AKMW}, tells us that 
an indecomposable symmetric cellular algebra of finite type is a Brauer tree algebra whose Brauer tree is a straight line 
with at most one exceptional vertex. Thus, $B$ is a Brauer tree algebra whose Brauer tree is a straight line with at most 
one exceptional vertex. Further, it is derived equivalent to the star-shape Brauer tree algebra whose central vertex 
might have a multiplicity $m>1$, by \cite[Theorem 4.2]{R89}. 

Now, the Rickard star theorem also states that 
the derived equivalence classes of Brauer tree algebras are determined by the number of edges and the multiplicity of the exceptional vertex. 
The proof is by observing that the number of edges $e$ and the determinant of the Cartan matrix $em+1$ of a star-shape Brauer tree algebra are invariant under 
derived equivalence. 

Hence, we conclude that $B$ is a Brauer tree algebra whose Brauer tree is a straight line without exceptional vertex.

\subsection{}
Since Hecke algebras of type $D$ are cellular by \cite[Theorem 1.1]{G}, we know that $B$ is a Brauer tree algebra whose Brauer tree is a straight line 
with at most one exceptional vertex by the Ohmatsu's theorem. Thus, to prove that there is no exceptional vertex, it suffices to show that $\Rad^3 B=0$. 
Consider Hecke algebras ${\mathcal H}$ of type $B$ with $Q=1$. 
As is explained in Appendix 2, $B$ is a subalgebra of a block algebra $A$ of ${\mathcal H}$ such that 
an irreducible $A$-module restricts to either an irreducible $B$-module or direct sum of two irreducible $B$-modules, and 
$A$ is of finite representation type. We have already proved that $A$ is a Brauer tree algebra whose Brauer tree is a straight line without exceptional vertex. 
Thus, for any indecomposable projective $B$-module, which is a direct summand of an indecomposable projective $A$-module viewed as a $B$-module, 
the length of the radical series is three. Hence, there is no exceptional vertex in the Brauer tree associated with $B$.

\section{Appendix 1}

Let $e\ge2$ be the quantum characteristic of $q\in\bR^\times$. Then, the Hecke algebra of type $B$ are cyclotomic quiver Hecke algebras of level two 
in type $A^{(1)}_\ell$, where $\ell=e-1$. We consider the case $-Q\not\in q^{\Z}$ from the introduction. Recall from \cite[3.2]{EN02} that any block algebra of a Hecke algebra of type $A$ which has finite representation type has $\ell$ simple modules, $S_1,\dots,S_\ell$ say, and their projective covers are given by 
$$
P_1=\begin{array}{c} S_1 \\ S_2 \\ S_1 \end{array}, \quad
P_i=\begin{array}{c} S_i \\ S_{i-1}\oplus S_{i+1} \\ S_i \end{array}\;\text{($2\le i\le \ell-1$)}, \quad
P_{\ell}=\begin{array}{c} S_{\ell} \\ S_{\ell-1} \\ S_{\ell} \end{array}
$$
if $\ell\ge2$. If $\ell=1$ then it is Morita equivalent to $\bR[x]/(x^2)$. 

Block algebras of tame representation type appear only when $\ell=1$ and they have two simple modules. Let $S$ and $T$ be the simple modules. Then 
the heart of the projective covers are either $H(P(S))=S \oplus T=H(P(T))$ or
$$
H(P(S))=S\oplus \begin{array}{c} T \\ S \\ T \end{array}, \quad
H(P(T))=\begin{array}{c} S \\ T \\ S \end{array}.
$$

Suppose that $B_1$ and $B_2$ are block algebras of Hecke algebras of type $A$. If $B_1$ and $B_2$ are of finite representation type, then 
we know the representation type of $B_1\otimes B_2$ by \cite[Lem.67]{Ar05}. It is of wild representation type if $\ell\ge2$ and it is of tame representation type if $\ell=1$. 

If $B_1$ is of finite representation type and $B_2$ is of tame representation type, then the quiver of $B_1\otimes B_2$ has two vertices $1$ and $2$ such that 
there are arrows $1\rightarrow 2$, $1\leftarrow 2$ and two loops on each vertex. Namely, it is the quiver which appeared in the proof of 
Proposition \ref{ell=1} and Proposition \ref{i=2}. 

If both $B_1$ and $B_2$ are of tame representation type, then the quiver of $B_1\otimes B_2$ has a connected subquiver which has two vertices and two loops on one of the vertices. 
Thus, they are of wild representation type.

\section{Appendix 2}

We assume that $\bR$ has an odd characteristic and consider block algebras of Hecke algebras in type $D$. As is explained in \cite[4.4]{Ar05}, we embed the Hecke algebra of 
type $D$ and of rank $n$ to the Hecke algebra of type $B$ with unequal parameter $Q=1$ of the same rank. It is generated by $T_0,\dots,T_{n-1}$ and their relations are
\begin{gather*}
T_0^2=1, \quad (T_i-q)(T_i+1)=0\;\;(1\le i<n)\\
T_0T_1T_0T_1=T_1T_0T_1T_0, \quad T_iT_{i+1}T_i=T_{i+1}T_iT_{i+1}\;\;(1\le i<n-1)\\
T_iT_j=T_jT_i\;\;(j\ne i\pm1).
\end{gather*}
Then, we have an involutive algebra automorphism $\sigma: T_0\mapsto -T_0$, $T_i\mapsto T_i$, for $i\ne 0$, and the Hecke algebra of type $D$ is the fixed point 
subalgebra of $\sigma$. We have another involutive algebra automorphism $\tau: T_1\mapsto T_0T_1T_0$, $T_i\mapsto T_i$, for $i\ne 1$. Note that
$\tau$ is the inner automorphism induced by $T_0$ and $\sigma\tau=\tau\sigma$ holds.

Let $A$ be a block algebra of the Hecke algebra of type $B$. Then $A$ covers one or two blocks of the Hecke algebra of type $D$. 
Let $B$ be the (direct sum of)  block algebra(s) covered by $A$, and we consider the induction and restriction functors. Then,
\begin{itemize}
\item
$\Res^A_B\circ\Ind^A_B(M)\simeq M\oplus M^\tau$, where $b\in B$ acts on $M^\tau$ by $\tau(b)$, for a $B$-module $M$.
\item
$\Ind^A_B\circ\Res^A_B(M)\simeq M^{\oplus 2}$, for an $A$-module $M$.
\item
If $B$ splits into two block algebras $B'$ and $B''$, then $M$ is a $B'$-module if and only if $M^\tau$ is a $B''$-module.
\end{itemize}
Using these simple facts, we know that $A$ and $B$, $B'$, $B''$ have the same representation type. 

\section{Appendix 3}

In this appendix, I explain computations from \cite{Ka15} to prove Theorem B. Hence, it suffices to consider $\fqHH(\beta)$ with 
$\beta=\lambda^0_i+k\delta$, for $0\le i\le \frac{\ell+1}{2}$ and $k\in\Z_{\ge0}$. 

\begin{lem}\label{formula1}
Let $\nu\in I^n$ and suppose that $\nu_i\ne \nu_{i+1}$. If $x_i^ke(s_i\nu)=cx_i^{k-1}x_je(s_i\nu)$, for some $j<i$ and $c\in \bR$, then 
$$
x_{i+1}^k\psi_i^2e(\nu)=cx_{i+1}^{k-1}x_j\psi_i^2e(\nu).
$$
\end{lem}
\begin{proof}
Note that $x_{i+1}^k\psi_i^2e(\nu)=x_{i+1}^k\psi_ie(s_i\nu)\psi_i=\psi_ix_i^ke(s_i\nu)\psi_i$. We replace $x_i^ke(s_i\nu)$ with $cx_i^{k-1}x_je(s_i\nu)$ to obtain the result.
\end{proof}

\begin{lem}\label{formula2}
Let $\nu\in I^n$ and suppose that $\nu_i=\nu_{i+2}\ne \nu_{i+1}$. If $x_{i+1}e(s_{i+1}\nu)=cx_je(s_{i+1}\nu)$, for some $j\le i$ and $c\in\bR$, then 
$$
x_{i+2}\psi_{i+1}\psi_i\psi_{i+1}e(\nu)=cx_j\psi_{i+1}\psi_i\psi_{i+1}e(\nu).
$$
\end{lem}
\begin{proof}
Noting that $s_is_{i+1}\nu=s_{i+1}\nu$, we rewrite the left hand side to $\psi_{i+1}x_{i+1}e(s_{i+1}\nu)\psi_i\psi_{i+1}$. Then we replace 
$x_{i+1}e(s_{i+1}\nu)$ with $cx_je(s_{i+1}\nu)$.
\end{proof}

\begin{lem}\label{formula3}
Let $\nu\in I^n$ and suppose that $\nu_i=\nu_{i+1}$. Then, $x_i^2e(\nu)=0$ implies 
$$
x_{i+1}e(\nu)=-x_ie(\nu).
$$
\end{lem}
\begin{proof}
As $\nu_i=\nu_{i+1}$, we have $\psi_i^2e(\nu)=0$ and $(\psi_ix_{i+1}-x_i\psi_i)e(\nu)=e(\nu)$. We start with
$$
x_ie(\nu)= x_i(\psi_ix_{i+1}-x_i\psi_i)e(\nu)=x_i\psi_ix_{i+1}e(\nu)-x_i^2e(\nu)\psi_i=x_i\psi_ix_{i+1}e(\nu).
$$
Then we deduce
\begin{align*}
x_{i+1}e(\nu)&=(\psi_ix_{i+1}-x_i\psi_i)e(\nu)x_{i+1}=\psi_ix_{i+1}^2e(\nu)-x_i\psi_ix_{i+1}e(\nu)\\
&=\psi_i(x_{i+1}e(\nu))x_{i+1}-x_ie(\nu),
\end{align*}
so that
\begin{align*}
x_{i+1}e(\nu)&= \psi_i(\psi_ix_{i+1}^2e(\nu)-x_ie(\nu))x_{i+1}-x_ie(\nu)\\
&=\psi_i^2x_{i+1}^3e(\nu)-\psi_i(x_{i+1}e(\nu))x_i-x_ie(\nu)\\
&=\psi_i^2x_{i+1}^3e(\nu)-\psi_i(\psi_ix_{i+1}^2e(\nu)-x_ie(\nu))x_i-x_ie(\nu)\\
&=\psi_i^2e(\nu)x_{i+1}^3-\psi_i^2e(\nu)x_ix_{i+1}^2+\psi_ix_i^2e(\nu)-x_ie(\nu).
\end{align*}
Using $\psi_i^2e(\nu)=0$ and $x_i^2e(\nu)=0$, we obtain the result.
\end{proof}

\begin{prop}\label{i=2 for s=0}
Suppose $\ell\ge3$. Then $\fqHH(\lambda^0_2)$ is wild if ${\rm char}\bR=2$, tame if ${\rm char}\bR\ne2$.
\end{prop}
\begin{proof}
Let $e_1=e(010\ell)$ and $e_2=e(01\ell0)$. If $e_1x_1^{a_1}x_2^{a_2}x_3^{a_3}x_4^{a_4}\psi_we_1\ne0$, for $a_1,a_2,a_3,a_4\ge0$ and $w\in S_4$, 
then either $\psi_w=1$ or $\psi_w=\psi_1\psi_2\psi_1$. However, $\psi_1e_1=e(100\ell)\psi_1=0$ implies that the latter does not occur. 
We show that $e_1\fqHH(\lambda^0_2)e_1$ has the basis 
$$
\{e_1, e_1x_1e_1, e_1x_4e_1, e_1x_1x_4e_1\}.
$$
First, $\psi_1^2e_1=0$ implies $x_2e_1=-x_1e_1$. Next, we apply Lemma \ref{formula3} to $x_1^2e(001\ell)=0$ and obtain $x_2e(001\ell)=-x_1e(001\ell)$. Then Lemma \ref{formula2} implies
$x_3\psi_2\psi_1\psi_2e_1=-x_1\psi_2\psi_1\psi_2e_1$ and we may conclude $x_3e_1=-x_1e_1$ because
$$
\psi_2\psi_1\psi_2e_1=(\psi_2\psi_1\psi_2-\psi_1\psi_2\psi_1)e_1=e_1.
$$
Finally, $\psi_1e(0\ell10)=0$ implies $x_2e(0\ell10)=-\lambda x_1e(0\ell10)$ by $\psi_1^2e(0\ell10)=0$. Thus, applying Lemma \ref{formula1}, we obtain
$x_3e_2=x_3\psi_2^2e_2=-\lambda x_1\psi_2^2e_2=-\lambda x_1e_2$. 
Applying Lemma \ref{formula1} again, we obtain $x_4\psi_3^2e_1=-\lambda x_1\psi_3^2e_1$, that is,
$$
x_4(\lambda x_3+x_4)e_1=-\lambda x_1(\lambda x_3+x_4)e_1,
$$
and $x_4^2e_1=\lambda^2x_1^2e_1=0$ follows. 

Noting that $(x_1+x_2)e_2=\psi_1^2e_2=0$ and that Lemma \ref{formula1} applied to $x_3e_1=-x_1e_1$ implies $x_4\psi_3^2e_2=-x_1\psi_3^2e_2$, so that 
$x_4^2e_2=0$ follows,  we also deduce that $e_2\fqHH(\lambda^0_2)e_2$ has the basis 
$$
\{e_2, e_2x_1e_2, e_2x_4e_2, e_2x_1x_4e_2\}.
$$
If $e_1\psi_we_2\ne0$ then $\psi_w=\psi_3$. Further, $x_2\psi_3e_2=\psi_3x_2e_2=-x_1\psi_3e_2$ and
$$
x_3\psi_3e_2=x_3e_1\psi_3=-x_1\psi_3e_2, \;\; x_4\psi_3e_2=\psi_3x_3e_2=-\lambda x_1\psi_3e_2
$$
imply that $e_1\fqHH(\lambda^0_2)e_2$ has the basis $\{e_1\psi_3e_2, e_1x_1\psi_3e_2\}$. Then, 
$e_2\fqHH(\lambda^0_2)e_1$ has the basis $\{e_2\psi_3e_1, e_2x_1\psi_3e_1\}$. Thus, we have obtained explicit basis of $e\fqHH(\lambda^0_2)e$, for $e=e_1+e_2$. 
As the number of simple $\fqHH(\lambda^0_2)$-modules is two, $A$ is the basic algebra of $\fqHH(\lambda^0_2)$. 

Let $\alpha=x_1e_1$, $\beta=x_1e_2$, $\mu=\psi_3e_2$ and $\nu=\psi_3e_1$. Then, it is easy to see that 
$A$ is isomorphic to the path algebra defined by the quiver with two vertices $1$ and $2$, 
a loop $\alpha$ on the vertex $1$, a loop $\beta$ on the vertex $2$, and arrows $\mu$ and $\nu$ from vertex $1$ to vertex $2$ and vertex $2$ to vertex $1$,
modulo the ideal generated by
$$
\alpha^2,\; \beta^2,\; \alpha\mu-\mu\beta,\; \beta\nu-\nu\alpha,\; \mu\nu\mu+2\lambda\alpha\mu,\; \nu\mu\nu+2\lambda\beta\nu.
$$
If ${\rm char}{\bR}=2$, we consider the algebra with the additional relation $\nu\alpha=0$ to conclude that $A$ is wild. If ${\rm char}\bR\ne2$, we replace $\alpha$ and $\beta$ 
with $\alpha'=\mu\nu+2\lambda\alpha$ and $\beta'=\nu\mu+2\lambda\beta$ and observe that $A$ is a special biserial algebra. 
\end{proof}

\begin{prop}\label{i=3 for s=0}
Suppose $\ell\ge5$. Then $\fqHH(\lambda^0_3)$ is wild.
\end{prop}
\begin{proof}
We consider $A=e\fqHH(\lambda^0_3)e$ with $e=e(\nu)$ where $\nu=(0\;1\;2\;\ell\;0\;1\;\ell-1\;\ell\;0)$. Our goal is to show that 
$A$ has the basis $\{ x_1^ax_5^bx_9^ce \mid a,b,c=0,1\}$ and $x_1^2e=x_5^2e=x_9^2e=0$. 

If $e\psi_we\ne0$, we may conclude that $\psi_w=1$. The relations $\psi_1^2e=\psi_2^2e=0$ imply $x_2e=-x_1e$ and $x_3e=x_1e$. Applying Lemma \ref{formula1} twice to
$\psi_1^2e(s_2s_3\nu)=(\lambda x_1+x_2)e(s_2s_3\nu)=0$, we obtain $x_3e(s_3\nu)=-\lambda x_1e(s_3\nu)$ and then $x_4e=-\lambda x_1e$. 

Lemma \ref{formula3} implies $x_2e(s_2s_3s_4\nu)=-x_1e(s_2s_3s_4\nu)$. Using $\psi_1e(s_3s_4\nu)=0$, we can show 
$\psi_2\psi_1\psi_2e(s_3s_4\nu)=e(s_3s_4\nu)$. Thus, using Lemma \ref{formula2}, we obtain
$$
x_3e(s_3s_4\nu)=-x_1e(s_3s_4\nu).
$$
Now we apply Lemma \ref{formula1} to this equality and conclude $x_5^2e=0$. $x_6e=-x_5e$ follows from $\psi_5^2e=0$. Using $e(s_4s_5s_6\nu)=0$, we deduce 
$x_4e(s_5s_6\nu)+x_5e(s_5s_6\nu)=\psi_4^2e(s_5s_6\nu)=0$. Then, applying Lemma \ref{formula1} twice, we have $x_7e=-x_4e=\lambda x_1e$. Next we consider $x_8e$. 
\begin{align*}
x_8e&= x_8e(0\;1\;2\;\ell\;0\;1\;\ell-1\;\ell\;0)\\
&=x_8\psi_6^2e(0\;1\;2\;\ell\;0\;1\;\ell-1\;\ell\;0)\\
&=\psi_6x_8e(0\;1\;2\;\ell\;0\;\ell-1\;1\;\ell\;0)\psi_6\\
&=\psi_6x_8\psi_7^2e(0\;1\;2\;\ell\;0\;\ell-1\;1\;\ell\;0)\psi_6\\
&=\psi_6\psi_7x_7e(0\;1\;2\;\ell\;0\;\ell-1\;\ell\;1\;0)\psi_7\psi_6\\
&=\psi_6\psi_7x_7\psi_5^2e(0\;1\;2\;\ell\;0\;\ell-1\;\ell\;1\;0)\psi_7\psi_6\\
&=\psi_6\psi_7\psi_5x_7e(0\;1\;2\;\ell\;\ell-1\;0\;\ell\;1\;0)\psi_5\psi_7\psi_6.
\end{align*}
Now $e(0\;1\;2\;\ell\;\ell-1\;\ell\;0\;1\;0)=0$ implies $\psi_6^2e(0\;1\;2\;\ell\;\ell-1\;0\;\ell\;1\;0)=0$, and it follows that 
$$
x_8e=\psi_6\psi_7\psi_5(-\lambda x_6)e(0\;1\;2\;\ell\;\ell-1\;0\;\ell\;1\;0)\psi_5\psi_7\psi_6=-\lambda x_5e.
$$

Before continuing computation, we prepare $x_4e(s_6s_7s_8\nu)=-\lambda x_1e(s_6s_7s_8\nu)$: to prove it, apply Lemma \ref{formula1} twice to 
$(\lambda x_1+x_2)e(s_2s_3s_6s_7s_8\nu)=\psi_1^2e(s_2s_3s_6s_7s_8\nu)=0$. 

Lemma \ref{formula3} implies $x_2e(s_2s_3s_4s_6s_7s_8\nu)=-x_1e(s_2s_3s_4s_6s_7s_8\nu)$. As $\psi_1e(s_3s_4s_6s_7s_8\nu)=0$, 
$\psi_2\psi_1\psi_2e(s_3s_4s_6s_7s_8\nu)=e(s_3s_4s_6s_7s_8\nu)$ holds, and we have
$$
x_3e(s_3s_4s_6s_7s_8\nu)=-x_1e(s_3s_4s_6s_7s_8\nu)
$$
by Lemma \ref{formula2}. Now we apply Lemma \ref{formula1} twice:
\begin{align*}
x_4e(s_4s_6s_7s_8\nu)&=-x_1e(s_4s_6s_7s_8\nu),\\
x_5(x_4+\lambda x_5)e(s_6s_7s_8\nu)&=-x_1(x_4+\lambda x_5)e(s_6s_7s_8\nu).
\end{align*}
Then, we use  $x_4e(s_6s_7s_8\nu)=-\lambda x_1e(s_6s_7s_8\nu)$ and conclude that $x_5^2e(s_6s_7s_8\nu)=0$. It follows that 
$x_6e(s_6s_7s_8\nu)=-x_5e(s_6s_7s_8\nu)$ by Lemma \ref{formula3}. 

It is easy to see that $e(s_5s_7s_8\nu)=0$, so that $\psi_6\psi_5\psi_6e(s_7s_8\nu)=e(s_7s_8\nu)$. Hence, applying Lemma \ref{formula2} to 
$x_6e(s_6s_7s_8\nu)=-x_5e(s_6s_7s_8\nu)$, we deduce $x_7e(s_7s_8\nu)=-x_5e(s_7s_8\nu)$. 

We apply Lemma \ref{formula1} to $x_7e(s_7s_8\nu)=-x_5e(s_7s_8\nu)$. Then,
\begin{align*}
x_8e(s_8\nu)&=-x_5e(s_8\nu),\\
x_9(x_8+\lambda x_9)e&=-x_5(x_8+\lambda x_9)e.
\end{align*}
As $x_8e=-\lambda x_5e$ and $x_5^2e=0$, we have $x_9^2e=0$. 
\end{proof}

By Proposition \ref{i=3 for s=0} and Lemma \ref{wildness propagation 2}, we have the following corollary. 
\begin{cor}\label{general i for s=0}
Suppose $\ell\ge5$. Then, $\fqHH(\lambda^0_i)$ is wild, for $3\le i\le \frac{\ell+1}{2}$.
\end{cor}

\begin{prop}\label{delta for s=0 and l=1}
Suppose $\ell=1$. Then $\fqHH(\delta)$ is tame.
\end{prop}
\begin{proof}
$\fqHH(\delta)={\rm Span}\{e, x_1e, x_2e, x_1x_2e\}$ and $x_1^2e=0, x_2^2e=-\lambda x_1x_2e$, where $e=e(01)$.
\end{proof}

\begin{prop}\label{delta for s=0}
For $\ell\ge2$, $\fqHH(\delta)$ is wild if $\lambda=(-1)^{\ell-1}$, tame if $\lambda\ne(-1)^{\ell-1}$.
\end{prop}
\begin{proof}
First we suppose $\ell=2$, and define $e_1=e(012)$ and $e_2=e(021)$. Then $e_1+e_2=1$ and $x_2e_1=-x_1e_1$, $x_2e_2=-\lambda x_1e_2$. Applying Lemma \ref{formula1}, we have
$x_3\psi_2^2e_1=-\lambda x_1\psi_2^2e_1$ and $x_3\psi_2^2e_2=-x_1\psi_2^2e_2$, so that
$$
x_3^2e_1=(1-\lambda)x_1x_3e_1,\;\; x_3^2e_2=0.
$$
We may also compute $x_3\psi_2e_1=-x_1\psi_2e_1$, $x_3\psi_2e_2=-\lambda x_1\psi_2e_2$. 

Define $\alpha=x_1e_1, \beta=x_1e_2$, $\mu=\psi_2e_2$ and $\nu=\psi_2e_1$. Then, we are in the similar situation as in the proof of Proposition \ref{i=2 for s=0}. 
We have the same quiver but the defining relations here are given by
$$
\alpha^2=\beta^2=0,\; \alpha\mu=\mu\beta, \; \beta\nu=\nu\alpha, \; \mu\nu\mu+(1+\lambda)\alpha\mu=\nu\mu\nu+(1+\lambda)\beta\nu=0.
$$
If $\lambda=-1$ then it is wild. If $\lambda\ne-1$, we use $\alpha'=\mu\nu+(1+\lambda)\alpha$ and $\beta'=\nu\mu+(1+\lambda)\beta$ instead of $\alpha$ and $\beta$ 
to show that it is a special biserial algebra.

Next we suppose $\ell\ge3$. 
Let $\nu_i=(0\;1\;2\;\cdots\; i \; \ell \; \ell-1 \; \cdots \; i+1)$ and define $e_i=e(\nu_i)$. We may prove that 
$x_ke_i$ is a scalar multiple of $x_1e_i$, for $1\le k\le\ell$, by induction on $k$. Applying Lemma \ref{formula1} to $x_\ell e(s_\ell\nu_i)=c_ix_1e(s_\ell \nu_i)$, for 
some $c_i\in\bR$, we know that $x_{\ell+1}^2e_i$ is a scalar multiple of $x_1x_{\ell+1}e_i$. If $e_i\psi_we_i\ne0$ then $\psi_w=1$. Thus $e_i\fqHH(\delta)e_i={\rm Span}\{e_i, x_1e_i, x_{\ell+1}e_i, x_1x_{\ell+1}e_i\}$. 

The similar computation shows that
\begin{align*}
e_i\fqHH(\delta)e_{i+1}&={\rm Span}\{\psi_\ell\psi_{\ell-1}\cdots\psi_{i+2}e_{i+1}, x_1\psi_\ell\psi_{\ell-1}\cdots\psi_{i+2}e_{i+1}\},\\
e_{i+1}\fqHH(\delta)e_i&={\rm Span}\{\psi_{i+2}\psi_{i+3}\cdots\psi_\ell e_i, x_1\psi_{i+2}\psi_{i+3}\cdots\psi_\ell e_i\}.
\end{align*}
Let $e=\sum_{i=0}^{\ell-1} e_i$ and $A=e\fqHH(\delta)e$. Then, $A$ is the basic algebra of $\fqHH(\delta)$. The quiver of $A$ has vertices $0,1,\dots,\ell-1$. 
By degree consideration, $\dim e_0(\Rad A/\Rad^2 A)e_0=1$ because nonzero scalar multiples of 
$$
(\psi_\ell\psi_{\ell-1}\cdots \psi_2e_1)(\psi_2\psi_3\cdots\psi_\ell e_0)=\psi_\ell^2e_0=(x_\ell+x_{\ell+1})e_0
$$
are the only degree two elements in $e_0(\Rad A)e_0$. Similarly, $\dim e_{\ell-1}(\Rad A/\Rad^2 A)e_{\ell-1}=1$ because
$(\psi_\ell e_{\ell-2})(\psi_\ell e_{\ell-1})=\psi_\ell^2e_{\ell-1}=(x_\ell+x_{\ell+1})e_{\ell-1}$. 

If $i\ne 0, \ell-1$, $\dim e_i(\Rad A/\Rad^2 A)e_i=0$ or $1$ depending on whether $(x_\ell+x_{\ell+1})e_i$ and $(x_{i+1}+x_{\ell+1})e_i$ are linearly independent or not, because 
degree two elements in $e_i(\Rad A)e_i$ are linear combination of the following elements:
\begin{align*}
(\psi_\ell\psi_{\ell-1}\cdots \psi_{i+2}e_{i+1})(\psi_{i+2}\psi_{i+3}\cdots\psi_\ell e_i)&=\psi_\ell^2e_i=(x_\ell+x_{\ell+1})e_i,\\
(\psi_{i+1}\psi_{i+2}\cdots\psi_\ell e_{i-1})(\psi_\ell\psi_{\ell-1}\cdots \psi_{i+1}e_i)&=(x_{i+1}+x_{\ell+1})e_i.
\end{align*}
Note that $e(s_j\nu_i)=0$, for $1\le j\le i$, implies $\psi_j^2e_i=0$ and $x_{i+1}e_i=(-1)^ix_1e_i$ holds. Similarly, 
$e(s_j\nu_i)=0$, for $i+2\le j\le \ell-1$, implies $x_\ell e_i=(-1)^{\ell-i-2}x_{i+2}e_i$. On the other hand, if we repeatedly apply Lemma \ref{formula1} to
$$
(\lambda x_1+x_2)e(s_2s_3\cdots s_{i+1}\nu_i)=\psi_1^2e(s_2s_3\cdots s_{i+1}\nu_i)=0,
$$
we reach $x_{i+2}e_i=-\lambda x_1e_i$. Therefore, $x_\ell e_i=(-1)^{\ell-i-2}x_{i+2}e_i=(-1)^{\ell-i-1}\lambda x_1e_i$ and we compare it with  $x_{i+1}e_i=(-1)^ix_1e_i$. 
Given the explicit shape of the quiver of $A$, it is clear that $A$ is wild if $\lambda=(-1)^{\ell-1}$. As we are given explicit basis elements of $A$, 
we can show that $A$ is a special biserial algebra if $\lambda\ne(-1)^{\ell-1}$.
\end{proof}

\begin{prop}
$\fqHH(\lambda^0_1+\delta)$ is wild.
\end{prop}
\begin{proof}
We define $e=e(010)$ for $\ell=1$ and $e=e(0120)$ for $\ell=2$. If $\ell\ge3$, we define $e=e(\nu)$, where $\nu=(0\; 1\; 2\;\cdots \; \ell-1 \; \ell \; 0)$. Then we 
consider $A=e\fqHH(\lambda^0_1+\delta)e$. Explicit computation of the algebra structure shows that 
$$
A={\rm Span}\{x_1^ax_{\ell+1}^bx_{\ell+2}^ce \mid a,b,c=0,1\}.
$$
For example, we remark that the subalgebra of $e(\delta,0)\fqHH(\lambda^0_1+\delta)e(\delta,0)$ generated by 
$\psi_ke(\delta,0)$, for $1\le k\le \ell$, and $x_ie(\delta,0)$, for $i=1,\ell+1$, may be viewed as a quotient algebra of $\fqHH(\delta)$. 
Thus, we know that $x_ie(\nu,0)$ is a scalar multiple of $x_1e(\nu,0)$, for $\nu\in I^\delta$ and $1\le i\le \ell$. So the computation is 
to show that $x_{\ell+1}^2e$ and $x_{\ell+2}^2e$ are linear combination of $x_1^ax_{\ell+1}^bx_{\ell+2}^ce$ with  $a,b,c=0,1$. As $\dim A=8$ and 
the computation is not too difficult, we omit the details. 
\end{proof}

\begin{cor}
If $0\le i\le \frac{\ell+1}{2}$ and $k\ge1$ satisfy $i+k\ge2$, then $\fqHH(\lambda^0_i+k\delta)$ is wild.
\end{cor}

The above corollary tells that if $\fqHH(\lambda^0_i+k\delta)$ is not wild, then either $k=0$, or $k=1$ and $i=0$. Suppose $k=0$. 
It is clear that $\fqHH(\lambda^0_1)$ is finite. For $i\ge2$, $\fqHH(\lambda^0_i)$ are treated in 
Proposition \ref{i=2 for s=0} and Proposition \ref{i=3 for s=0}, Corollary \ref{general i for s=0}. Suppose $k=1$ and $i=0$. Then 
$\fqHH(\delta)$ are treated in Proposition \ref{delta for s=0 and l=1} and Proposition \ref{delta for s=0}. 


\bibliographystyle{amsplain}

\end{document}